\documentclass[review]{elsarticle}

\usepackage{lineno,hyperref}
\modulolinenumbers[5]

\journal{Australasian Journal of Combinatorics}

\usepackage{tikz}
 \newtheorem{theorem}{Theorem}
 \newtheorem{definition}{Definition}

 \newenvironment{proof}[1][Proof]{\textbf{#1.} }{\ \rule{0.5em}{0.5em}\vskip 12pt}

\begin{document}

\begin{frontmatter}

\title{A note on the restricted arc connectivity of oriented graphs of girth four}


\author[mymainaddress]{D. Gonz\'alez-Moreno\fnref{myfootnote}}\fntext[myfootnote]{This research was supported by CONACyT-M\'exico, under project CB-222104.}
\ead{dgonzalez@correo.cua.uam.mx} \author[mysecondaryaddress]{R. Hern\'andez Ortiz}

\ead{rangel@ciencias.unam.mx}



\address[mymainaddress]{Departamento de Matem\'aticas Aplicadas y Sistemas\\ Universidad Aut\'onoma Metropolitana UAM-Cuajimalpa\\ Ciudad de M\'exico, M\'exico}
\address[mysecondaryaddress]{Facultad de Ciencias\\ Universidad Nacional Aut\'onoma de M\'exico\\Ciudad de M\'exico, M\'exico}

\begin{abstract}
Let $D$ be a strongly connected digraph.
 An arc set $S$ of $D$ is a \emph{restricted arc-cut} of $D$ if $D-S$ has a non-trivial strong component $D_{1}$ such that $D-V(D_{1})$ contains an arc. 
The \emph{restricted arc-connectivity} $\lambda'(D)$ of a digraph $D$ is the minimum cardinality over all restricted arc-cuts of $D$. 
A strongly connected digraph $D$ is \emph{$\lambda'$-connected} when $\lambda'(D)$ exists.
This paper presents a family $\cal{F}$ of strong digraphs of girth four that are not $\lambda'$-connected and for every strong digraph $D\notin \cal{F}$ with girth four it follows that it is $\lambda'$-connected. Also, an upper and lower bound for $\lambda'(D)$ are given.
\end{abstract}

\begin{keyword}
 Strong connectivity \sep restricted arc-connectivity\sep girth
\MSC[2010] 05C40\sep  05C20
\end{keyword}

\end{frontmatter}

\section{Terminology and introduction}

All the digraphs considered in this work are finite oriented graphs, that is, a digraph with no symmetric arcs or loops. 
Let $D$ be a digraph with  vertex set $V(D)$ and  arc set  $A(D)$.  If $v$ is a vertex of  $D$, the  sets of \emph{out-neighbors} and \emph{in-neighbors} of $v$ are denoted by $N^+(v)$ and $N^-(v)$, respectively.
If $(u,v)$ is an arc of $D$, then it is said that $u$  \emph{dominates} $v$ (or $v$ is \emph{dominated by} $u$) and denote it by $u\rightarrow v$.   Two  vertices $u$ and $v$ of a digraph are   \emph{adjacent} if 
$u\rightarrow v$ or $v\rightarrow u$.
The numbers $d^+(v)= |N^+(v)|$  and $d^-(u) = |N^-(u)|$ are the \emph{out-degree} and the \emph{in-degree} of the vertex $v$.
 A \emph{$p$-cycle} is a cycle of length $p$. The minimum integer $p$ for which $D$ has a $p$-cycle is the \emph{girth of D}; denoted by $g(D)$. 
Given a digraph $D$, the subdigraph of $D$ induced by a set of vertices
$X$ is denoted by $D[X]$. For any subset $S\subset A(D)$,  $D-S$ denotes the subdigraph obtained by
deleting all the arcs of $S$.  
 A digraph $D$ is  \emph{strongly connected} or simply \emph{strong} if for every pair $u, v$ of vertices
there exists a directed path from $u$ to $v$ in $D$. A \emph{strong component} of a digraph $D$ is a maximal induced subdigraph of $D$ which is strong.
A digraph $D$ is called \emph{$k$-arc-connected}  if for any set $S$ of at most $k-1$ arcs the subdigraph $D-S$ is strong. The \emph{arc-connectivity $\lambda(D)$} of a digraph $D$ is  the largest value of $k$ such that $D$ is $k$-arc-connected.
For a pair $X,Y$ of vertex sets of a digraph $D$, we define $(X,Y)=\{x\rightarrow y\in A(D): x\in X,y\in Y\}$. If $Y=X^c$ we write $\partial^+(X)$  or $\partial^-(Y)$.
Let $D$ be a digraph with girth $g$. If $C=(v_1,v_2,\dots ,v_g)$  is  a $g$-cycle of $D$, then let 
$$
\xi(C)=\min \left\{\sum_{i=1}^g d^+(v_i)-g,  \sum_{i=1}^g d^-(v_i)-g\right\}
$$
and 
$$
\xi(D)=\min\{\xi(C) : \mbox{$C$ is a $g$-cycle of $D$}\}.
$$

We follow the book of Bang-Jensen and Gutin \cite{BaGu09} for terminology and definitions not given here.

As is well known, a digraph is a mathematical object modeling networks.  
An important parameter in the study of networks is the fault tolerance:  it is desirable that if some nodes  (resp., links) are unable to work, the message can still be always transmitted.
There are measures that indicate the fault tolerance of a network (modeled  by a digraph $D$), for instance the arc-connectivity  of $D$ measure how easily and reliably a packet sent by a vertex can reach another vertex. 	
Since digraphs with the same arc-connectivity can have large differences in the fault tolerance of the corresponding networks, one might be  interested  in defining more refined reliability parameters in order to provide a more accurate measure of fault tolerance in networks than the arc-connectivity (see \cite{EH1989}). In this context, Volkmann \cite{Vo2007} introduced the concept of restricted arc-connectivity of a digraph, which is closely related to the similar concept of restricted edge-connectivity in graphs proposed by Esfahanian and Hakimi \cite{EH88}.

\begin{definition}[Volkmann \cite{Vo2007}]
Let $D$ be a strongly connected digraph. An arc set $S$ of $D$ is a \emph{restricted arc-cut} of $D$, if $D-S$ has a non-trivial strong component $D_1$ such that $D-V(D_1)$ contains an arc. The \emph{restricted arc-connectivity} $\lambda'(D)$ of $D$ is the miminum cardianlity over all restricted arc-cuts. A strongly connected digraph $D$ is said  to be \emph{$\lambda'$-connected} if $\lambda'(D)$ exists.
\end{definition}
Observe that $\lambda'(D)$ does nots exists for every digraph with less than $g(D)+2$ vertices. Volkmann \cite{Vo2007}  
proved that each strong digraph $D$ of order
$n\ge 4$ and girth $g(D) = 2$ or $g(D) = 3$ except some families of digraphs
is $\lambda'$-conncected and  satisfies $\lambda(D)\le \lambda'(D)\le \xi(D)$.
Moreover, he proved the following characterization.
\begin{theorem}\label{characterization}\cite{Vo2007}
A strongly digraph $D$ is $\lambda'$-connected if and only if $D$ has a girth $C$ such that $D-V(C)$ contains an arc.
\end{theorem}

Concerning the arc-restricted connectivity of digraphs, Meierling, Volkmann and Winzen \cite{MVW08} studied the restricted arc-connectivity of generalizations of tournaments.
Balbuena, Garc\'ia-V\'azquez, Hansberg and Montejano \cite{BaGHM12,BGHM12} studied the restricted arc connectivity for some families of digraphs and introduced the concept of super-$\lambda'$ digraphs.
Results on restricted arc-connectivity of digraphs can be found in, e. g. Balbuena and Garc\'ia-V\'azquez \cite{BG10}, Chen, Liu and Meng \cite{CLM09},  Gr\"uter,  Guo and Holtkamp \cite{GGH13}, Gr\'uter,  Guo,  Holtkamp and Ulmer \cite{GGHU13} and Wang and Lin \cite{SS08}.



%

\vspace{6cm}
%

In this paper we present a family $\cal{F}$ of strong digraphs of girth four that are not $\lambda'$-connected and for every strong digraph $D\notin \cal{F}$ with girth four it follows that it is $\lambda'$-connected and  $\lambda(D)\le \lambda'(D)\le \xi(D)$.

\section{Main result}

Let $D$ be a strong digraph of girth $4$. In this section it is proved that $D$ is $\lambda'$-connected
with exception of the case that $D$ is a member of the following seven families (see Figure 1).

 Let $H_{1}$ consists of the $4$-cycle $(u,v,w,z,u)$ and the following vertex sets 
 $A=\{a_1,a_{2},\dots ,a_{p} \}$, $B=\{b_{1},b_{2},\dots , b_{q}\}$, 
$C=\{c_{1},c_{2},\dots , c_{r}\}$ and $D=\{d_{1},d_{2},\dots , d_{s}\}$  such that 
$u\rightarrow a_{i}\rightarrow v$ for  $i\le i\le p$,  
$v\rightarrow b_{i}\rightarrow w$ for  $i\le i\le q$, $w\rightarrow c_{i}\rightarrow z$ for  $i\le i\le r$ and 
$z\rightarrow d_{i}\rightarrow u$ for  $i\le i\le s$. The cases that $A$, $B$, $C$ or $D$ are empty sets are also allowed.

 Let $H_2$ consists of the  $4$-cycles $(u,v,w,z,u)$ and $(u,v,w,x,u)$,  and the vertex sets $A=\{a_1,a_2,\dots , a_p\}$ and $B=\{b_1,b_2,\dots , b_q\}$ such that $w\rightarrow a_i \rightarrow u$ for $1\le i \le p$ and  $u \rightarrow b_i \rightarrow w $ for $1\le i \le q$. The cases that $A$, $B$ or $C$  are empty sets are also allowed.

 Let $H_3$ consists of  the $4$-cycles $(u,v,w,z,u)$ and $(u,v,w,x,u)$ and the vertex sets 
$A=\{a_1,a_2,\dots , a_p\}$,  $B=\{b_1,b_2,\dots , b_q\}$ and $C=\{c_1,c_,\dots , c_r\}$ such that $u\rightarrow a_i \rightarrow v$, for $1\le i \le p$, $v\rightarrow b_i \rightarrow w$ for $1\le i\le q$  and $w \rightarrow c_i \rightarrow u$ for $1\le i \le r$. The cases that $A$, $B$ or $C$  are empty sets are also allowed.

 Let $H_4$ consists of the $4$-cycles $(u,v,w,z,u)$ and $(u,v,w,x,u)$, a vertex $y$ such  that $u\rightarrow y \rightarrow w$ and $y$ is adjacent to $v$, and the vertex set
$A=\{a_1,a_2,\dots, a_p\}$  such that $w\rightarrow a_i\rightarrow u$ for $1\le i\le p$.
The case that $A$ is an empty set is  also admissible.

 Let $H_5$ consists of the $4$-cycles $(u,v,w,z,u)$ and $(u,v,w,x,u)$ such that $x$ is adjacent to $z$, and the vertex set $A=\{a_1,a_2,\dots ,a_p\}$ such that  $u\rightarrow a_i\rightarrow w$ for $1\le i \le p$.

  Let $H_6$ consists of the $4$-cycles $(u,v,w,z,u)$ and $(u,v,w,x,u)$ such that $x$ is adjacent to $z$,  and the vertex sets $A=\{a_1,a_2,\dots ,a_p\}$,  and $B=\{b_1,b_2,\dots , b_q\}$, $u\rightarrow a_i\rightarrow v$ for $1\le i \le p$ and $v\rightarrow b_i \rightarrow w$ for $1\le i \le q$. The cases that $A$ and $B$  are empty sets are also allowed.

Let $H_7$ consists of the $4$-cycles $(u,v,w,z,u)$ and $(u,v,w,x,u)$ such that $x$ is adjacent to $z$,  and a vertex $y$ adjacent to $v$ such that $u\rightarrow w$.

\begin{figure}[htp]\label{Families}
\begin{center}
\begin{tikzpicture}[scale=.7]
\clip(-6.077877686265985,-8.01817197406777) rectangle (10.53175204462694,7.927072567589536);
\draw [->,line width=1.2pt] (-3.4,4.) -- (-1.4,4.);
\draw [->,line width=1.2pt] (-1.4,4.) -- (0.6,4.);
\draw [->,line width=1.2pt] (0.6,4.) -- (2.6,4.);
\draw [rotate around={0.:(-0.295,5.52)},line width=1.2pt] (-0.295,5.52) ellipse (0.7252051092762395cm and 0.5002973620961438cm);
\draw [->,line width=1.2pt] (-1.019911939132348,5.505775770292485) -- (-3.4,4.);
\draw [->,line width=1.2pt] (0.43010054470410974,5.5284955091154915) -- (2.6,4.);
\draw [rotate around={0.:(-2.395,3.08)},line width=1.2pt] (-2.395,3.08) ellipse (0.7252051092762344cm and 0.5002973620961405cm);
\draw [rotate around={0.:(-0.355,3.08)},line width=1.2pt] (-0.355,3.08) ellipse (0.7252051092762346cm and 0.5002973620961405cm);
\draw [rotate around={0.:(1.565,3.1)},line width=1.2pt] (1.565,3.1) ellipse (0.7252051092762344cm and 0.5002973620961405cm);
\draw [->,line width=1.2pt] (-3.4,4.) -- (-2.9585802036617617,3.3948558660901504);
\draw [->,line width=1.2pt] (-1.8351682300997105,3.3980225398781347) -- (-1.4,4.);
\draw [->,line width=1.2pt] (-1.4,4.) -- (-0.9671882199509313,3.348206073985828);
\draw [->,line width=1.2pt] (0.22243028867528225,3.3826767178973807) -- (0.6,4.);
\draw [->,line width=1.2pt] (0.6,4.) -- (0.979735931471287,3.3954291282536526);
\draw [->,line width=1.2pt] (2.1825575570590514,3.3622823499842127) -- (2.6,4.);
\draw (-0.645,5.8) node[anchor=north west] {$D$};
\draw (-2.74,3.42) node[anchor=north west] {$A$};
\draw (-0.645,3.42) node[anchor=north west] {$B$};
\draw (1.38,3.42) node[anchor=north west] {$C$};
\draw (-0.6441559600167281,7.713520185335197) node[anchor=north west] {$H_1$};
\draw [shift={(-0.35819667527799626,3.992017297540155)},line width=1.2pt]  plot[domain=0.0026984964016016767:2.9138217335434025,variable=\t]({1.*2.9582074459314636*cos(\t r)+0.*2.9582074459314636*sin(\t r)},{0.*2.9582074459314636*cos(\t r)+1.*2.9582074459314636*sin(\t r)});
\draw [->,line width=1.2pt] (-3.24,4.66) -- (-3.4,4.);
\draw [->,line width=1.2pt] (3.8947899343412793,5.2012017362783505) -- (5.894789934341279,5.2012017362783505);
\draw [->,line width=1.2pt] (5.894789934341279,5.2012017362783505) -- (7.894789934341279,5.2012017362783505);
\draw [->,line width=1.2pt] (7.894789934341279,5.2012017362783505) -- (7.894789934341279,7.201201736278353);
\draw [->,line width=1.2pt] (7.894789934341279,7.201201736278353) -- (3.8947899343412793,5.2012017362783505);
\draw [->,line width=1.2pt] (3.8947899343412793,7.201201736278353) -- (7.894789934341279,5.2012017362783505);
\draw [->,line width=1.2pt] (3.8947899343412793,7.201201736278353) -- (3.8947899343412793,5.2012017362783505);
\draw [rotate around={0.:(5.7797899343412755,4.101201736278349)},line width=1.2pt] (5.7797899343412755,4.101201736278349) ellipse (0.7252051092762319cm and 0.5002973620961403cm);
\draw [rotate around={0.:(5.845940977814994,2.881238171120028)},line width=1.2pt] (5.845940977814994,2.881238171120028) ellipse (0.6856111324698931cm and 0.44095081921530155cm);
\draw (5.45,4.4) node[anchor=north west] {$A$};
\draw (5.45,3.2) node[anchor=north west] {$B$};
\draw [->,line width=1.2pt] (7.894789934341279,5.2012017362783505) -- (6.4969722284771185,4.1754133219181755);
\draw [->,line width=1.2pt] (5.080922285378787,4.234806023458734) -- (3.8947899343412793,5.2012017362783505);
\draw [->,line width=1.2pt] (3.8947899343412793,5.2012017362783505) -- (5.285398970749378,3.135144225934499);
\draw [->,line width=1.2pt] (6.411434508709915,3.1305612265282723) -- (7.894789934341279,5.2012017362783505);
\draw (5.691231380195288,7.879616482644128) node[anchor=north west] {$H_2$};
\draw [->,line width=1.2pt] (3.44282541052074,-0.2468986838154264) -- (5.442825410520739,-0.2468986838154264);
\draw [->,line width=1.2pt] (5.442825410520739,-0.2468986838154264) -- (7.442825410520739,-0.2468986838154264);
\draw [->,line width=1.2pt] (7.442825410520739,-0.2468986838154264) -- (7.442825410520739,1.7531013161845734);
\draw [->,line width=1.2pt] (7.442825410520739,1.7531013161845734) -- (3.44282541052074,-0.24689868381542635);
\draw [->,line width=1.2pt] (3.44282541052074,1.7531013161845734) -- (7.442825410520739,-0.24689868381542635);
\draw [->,line width=1.2pt] (3.44282541052074,1.7531013161845734) -- (3.44282541052074,-0.24689868381542635);
\draw [rotate around={0.:(5.447825410520738,-2.606898683815425)},line width=1.2pt] (5.447825410520738,-2.606898683815425) ellipse (0.7252051092762395cm and 0.5002973620961437cm);
\draw (5.12,-2.2) node[anchor=north west] {$A$};
\draw [->,line width=1.2pt] (3.44282541052074,-0.2468986838154264) -- (5.442825410520739,-1.646898683815426);
\draw [->,line width=1.2pt] (5.442825410520739,-1.646898683815426) -- (7.442825410520739,-0.24689868381542635);
\draw [line width=1.2pt,dash pattern=on 3pt off 3pt] (5.442825410520739,-0.2468986838154264)-- (5.442825410520739,-1.646898683815426);
\draw [->,line width=1.2pt] (-2.8325203285514284,-0.08935862342030845) -- (-0.45252032855143165,-0.14935862342030803);
\draw [->,line width=1.2pt] (-0.45252032855143165,-0.14935862342030803) -- (1.5474796714485688,-0.14935862342030803);
\draw [->,line width=1.2pt] (1.5474796714485688,-0.14935862342030803) -- (1.5474796714485688,1.8506413765796905);
\draw [->,line width=1.2pt] (1.5474796714485688,1.8506413765796907) -- (-2.8325203285514284,-0.08935862342030854);
\draw [->,line width=1.2pt] (-2.852520328551428,1.8106413765796916) -- (1.5474796714485688,-0.14935862342030792);
\draw [->,line width=1.2pt] (-2.852520328551428,1.8106413765796916) -- (-2.8325203285514284,-0.08935862342030854);
\draw [rotate around={0.:(0.48747967144856985,-1.1493586234203084)},line width=1.2pt] (0.48747967144856985,-1.1493586234203084) ellipse (0.6406632924203222cm and 0.4837865792421787cm);
\draw [rotate around={0.:(-1.3525203285514291,-1.1493586234203084)},line width=1.2pt] (-1.3525203285514291,-1.1493586234203084) ellipse (0.6406632924203232cm and 0.48378657924217816cm);
\draw [->,line width=1.2pt] (-2.8325203285514284,-0.08935862342030845) -- (-1.9691470546310228,-1.0180848639045117);
\draw [->,line width=1.2pt] (-0.7259745553684048,-1.0483572440818878) -- (-0.45252032855143165,-0.14935862342030803);
\draw [->,line width=1.2pt] (-0.45252032855143165,-0.14935862342030803) -- (-0.10633493250646653,-0.9677592470195437);
\draw [->,line width=1.2pt] (1.0476228385126682,-0.9145507236587856) -- (1.5474796714485688,-0.14935862342030803);
\draw (-0.6204279175440239,2.0899741193042725) node[anchor=north west] {$H_3$};
\draw [rotate around={0.:(-0.38752032855143065,-2.4293586234203075)},line width=1.2pt] (-0.38752032855143065,-2.4293586234203075) ellipse (0.6856111324698888cm and 0.44095081921529744cm);
\draw [shift={(0.29856471543683255,-0.9784202069980171)},line width=1.2pt]  plot[domain=-1.3768498625996122:0.5860330554556596,variable=\t]({1.*1.49904358732966*cos(\t r)+0.*1.49904358732966*sin(\t r)},{0.*1.49904358732966*cos(\t r)+1.*1.49904358732966*sin(\t r)});
\draw [shift={(-0.761562797483743,-0.35559064168643023)},line width=1.2pt]  plot[domain=3.157437716288981:4.566689538631108,variable=\t]({1.*2.131225064558987*cos(\t r)+0.*2.131225064558987*sin(\t r)},{0.*2.131225064558987*cos(\t r)+1.*2.131225064558987*sin(\t r)});
\draw [->,line width=1.2pt] (-2.892520328551427,-0.38935862342030836) -- (-2.8325203285514284,-0.08935862342030843);
\draw [->,line width=1.2pt] (0.5874796714485675,-2.449358623420307) -- (0.30747967144856636,-2.4693586234203067);
\draw (-1.65,-0.7765) node[anchor=north west] {$A$};
\draw (0.20,-0.7765) node[anchor=north west] {$B$};
\draw (-0.66,-2.05) node[anchor=north west] {$C$};
\draw [->,line width=1.2pt] (-5.326904618140781,-5.655774090714479) -- (-3.3269046181407798,-5.655774090714479);
\draw [->,line width=1.2pt] (-3.3269046181407798,-5.655774090714479) -- (-1.3269046181407815,-5.655774090714479);
\draw [->,line width=1.2pt] (-1.3269046181407815,-5.655774090714479) -- (-1.3269046181407815,-3.6557740907144787);
\draw [->,line width=1.2pt] (-1.3269046181407815,-3.6557740907144787) -- (-5.326904618140781,-5.655774090714479);
\draw [->,line width=1.2pt] (-5.326904618140781,-3.6557740907144787) -- (-1.3269046181407815,-5.655774090714479);
\draw [->,line width=1.2pt] (-5.326904618140781,-3.6557740907144787) -- (-5.326904618140781,-5.655774090714479);
\draw [rotate around={0.:(-3.4158294509854983,-6.73313851300022)},line width=1.2pt] (-3.4158294509854983,-6.73313851300022) ellipse (0.8509448065510048cm and 0.6407082516997279cm);
\draw (-3.66,-6.5) node[anchor=north west] {$A$};
\draw [->,line width=1.2pt] (-5.326904618140781,-5.655774090714479) -- (-4.2393013840170815,-6.571649586643957);
\draw [->,line width=1.2pt] (-2.565340909892856,-6.712159976323876) -- (-1.3269046181407815,-5.655774090714479);
\draw [line width=1.2pt,dash pattern=on 6pt off 6pt] (-5.326904618140781,-3.6557740907144787)-- (-1.3269046181407815,-3.6557740907144787);
\draw [->,line width=1.2pt] (0.6730953818592189,-5.655774090714479) -- (2.6730953818592225,-5.655774090714479);
\draw [->,line width=1.2pt] (2.6730953818592225,-5.655774090714479) -- (4.6730953818592225,-5.655774090714479);
\draw [->,line width=1.2pt] (4.6730953818592225,-5.655774090714479) -- (4.6730953818592225,-3.6557740907144787);
\draw [->,line width=1.2pt] (4.6730953818592225,-3.6557740907144787) -- (0.6730953818592189,-5.655774090714479);
\draw [->,line width=1.2pt] (0.6730953818592189,-3.6557740907144787) -- (4.6730953818592225,-5.655774090714479);
\draw [->,line width=1.2pt] (0.6730953818592189,-3.6557740907144787) -- (0.6730953818592189,-5.655774090714479);
\draw [rotate around={0.:(1.5780953818592194,-6.735774090714482)},line width=1.2pt] (1.5780953818592194,-6.735774090714482) ellipse (0.7252051092762395cm and 0.5002973620961473cm);
\draw (1.3,-6.4) node[anchor=north west] {$A$};
\draw (3.4,-6.4) node[anchor=north west] {$B$};
\draw [->,line width=1.2pt] (0.6730953818592189,-5.655774090714479) -- (1.2308162848269943,-6.29657017540478);
\draw [rotate around={0.:(3.5780953818592254,-6.715774090714483)},line width=1.2pt] (3.5780953818592254,-6.715774090714483) ellipse (0.7252051092762197cm and 0.5002973620961328cm);
\draw [->,line width=1.2pt] (2.094831539522393,-6.384748662359989) -- (2.6730953818592225,-5.655774090714479);
\draw [->,line width=1.2pt] (2.6730953818592225,-5.655774090714479) -- (2.942392881222606,-6.47500566902618);
\draw [->,line width=1.2pt] (4.213900086633112,-6.475134140995005) -- (4.6730953818592225,-5.655774090714479);
\draw [line width=1.2pt,dash pattern=on 6pt off 6pt] (0.6730953818592189,-3.6557740907144787)-- (4.6730953818592225,-3.6557740907144787);
\draw [->,line width=1.2pt] (6.015979730307655,-5.62951741398196) -- (8.015979730307663,-5.62951741398196);
\draw [->,line width=1.2pt] (8.015979730307663,-5.62951741398196) -- (10.01597973030766,-5.62951741398196);
\draw [->,line width=1.2pt] (10.01597973030766,-5.62951741398196) -- (10.01597973030766,-3.6295174139819597);
\draw [->,line width=1.2pt] (10.01597973030766,-3.6295174139819597) -- (6.015979730307655,-5.62951741398196);
\draw [->,line width=1.2pt] (6.015979730307655,-3.6295174139819597) -- (10.01597973030766,-5.62951741398196);
\draw [->,line width=1.2pt] (6.015979730307655,-3.6295174139819597) -- (6.015979730307655,-5.62951741398196);
\draw [line width=1.2pt,dash pattern=on 6pt off 6pt] (6.015979730307655,-3.6295174139819597)-- (10.01597973030766,-3.6295174139819597);
\draw [->,line width=1.2pt] (6.015979730307655,-5.62951741398196) -- (7.9928515000475935,-7.021594190214189);
\draw [->,line width=1.2pt] (7.9928515000475935,-7.021594190214189) -- (10.01597973030766,-5.62951741398196);
\draw (5.359038785577429,1.995061949413455) node[anchor=north west] {$H_4$};
\draw [->,line width=1.2pt] (7.442825410520739,-0.2468986838154264) -- (6.024493886309811,-2.3035315322771286);
\draw [->,line width=1.2pt] (4.814662746012289,-2.3629659317859946) -- (3.44282541052074,-0.24689868381542635);
\draw (-3.6813453965228633,-2.6319063327638883) node[anchor=north west] {$H_5$};
\draw (2.5116736888529276,-2.6319063327638883) node[anchor=north west] {$H_6$};
\draw (7.92166737262948,-2.821730672545523) node[anchor=north west] {$H_7$};
\draw [line width=1.2pt,dash pattern=on 6pt off 6pt] (8.015979730307663,-5.62951741398196)-- (7.9928515000475935,-7.021594190214189);
\draw [fill=black] (-3.4,4.) circle (2.0pt);
\draw[color=black] (-3.6576173540501586,4.391594239156592) node {$u$};
\draw [fill=black] (-1.4,4.) circle (2.0pt);
\draw[color=black] (-1.3322691917251492,4.439050324102) node {$v$};
\draw [fill=black] (0.6,4.) circle (2.0pt);
\draw[color=black] (0.7557985458728184,4.367866196683887) node {$w$};
\draw [fill=black] (2.6,4.) circle (2.0pt);
\draw[color=black] (2.7726821560526735,4.367866196683887) node {$z$};
\draw [fill=black] (3.8947899343412793,5.2012017362783505) circle (2.0pt);
\draw[color=black] (3.579435600124616,5.5305402778463995) node {$u$};
\draw [fill=black] (5.894789934341279,5.2012017362783505) circle (2.0pt);
\draw[color=black] (6.070880059758554,5.577996362791808) node {$v$};
\draw [fill=black] (7.894789934341279,5.2012017362783505) circle (2.0pt);
\draw[color=black] (8.064035627465707,5.577996362791808) node {$w$};
\draw [fill=black] (7.894789934341279,7.201201736278353) circle (2.0pt);
\draw[color=black] (8.064035627465707,7.571151930498972) node {$z$};
\draw [fill=black] (3.8947899343412793,7.201201736278353) circle (2.0pt);
\draw[color=black] (4.053996449578699,7.571151930498972) node {$x$};
\draw [fill=black] (3.44282541052074,-0.2468986838154264) circle (2.0pt);
\draw[color=black] (3.081146708197828,0.07309050912440467) node {$u$};
\draw [fill=black] (5.442825410520739,-0.2468986838154264) circle (2.0pt);
\draw[color=black] (5.620047252777175,0.14427463654251765) node {$v$};
\draw [fill=black] (7.442825410520739,-0.2468986838154264) circle (2.0pt);
\draw[color=black] (7.613202820484326,0.14427463654251765) node {$w$};
\draw [fill=black] (7.442825410520739,1.7531013161845734) circle (2.0pt);
\draw[color=black] (7.613202820484326,2.137430204249681) node {$z$};
\draw [fill=black] (3.44282541052074,1.7531013161845734) circle (2.0pt);
\draw[color=black] (3.60316364259732,2.137430204249681) node {$x$};
\draw [fill=black] (5.442825410520739,-1.646898683815426) circle (2.0pt);
\draw[color=black] (5.8,-1.8) node {$y$};
\draw [fill=black] (-2.8325203285514284,-0.08935862342030845) circle (2.0pt);
\draw[color=black] (-3.2305125895414837,0.3103709338514479) node {$u$};
\draw [fill=black] (-0.45252032855143165,-0.14935862342030803) circle (2.0pt);
\draw[color=black] (-0.2882353229261655,0.23918680643333495) node {$v$};
\draw [fill=black] (1.5474796714485688,-0.14935862342030803) circle (2.0pt);
\draw[color=black] (1.7049202447809857,0.23918680643333495) node {$w$};
\draw [fill=black] (1.5474796714485688,1.8506413765796907) circle (2.0pt);
\draw[color=black] (1.7049202447809857,2.232342374140498) node {$z$};
\draw [fill=black] (-2.852520328551428,1.8106413765796916) circle (2.0pt);
\draw[color=black] (-2.6847676126692877,2.1848862891950898) node {$x$};
\draw [fill=black] (-5.326904618140781,-5.655774090714479) circle (2.0pt);
\draw[color=black] (-5.579588794339198,-5.313175132179477) node {$u$};
\draw [fill=black] (-3.3269046181407798,-5.655774090714479) circle (2.0pt);
\draw[color=black] (-3.159328462123371,-5.265719047234068) node {$v$};
\draw [fill=black] (-1.3269046181407815,-5.655774090714479) circle (2.0pt);
\draw[color=black] (-1.1661728944162202,-5.265719047234068) node {$w$};
\draw [fill=black] (-1.3269046181407815,-3.6557740907144787) circle (2.0pt);
\draw[color=black] (-1.1661728944162202,-3.272563479526905) node {$z$};
\draw [fill=black] (-5.326904618140781,-3.6557740907144787) circle (2.0pt);
\draw[color=black] (-5.152484029830522,-3.272563479526905) node {$x$};
\draw [fill=black] (0.6730953818592189,-5.655774090714479) circle (2.0pt);
\draw[color=black] (0.3286937813641432,-5.360631217124886) node {$u$};
\draw [fill=black] (2.6730953818592225,-5.655774090714479) circle (2.0pt);
\draw[color=black] (2.843866283470786,-5.265719047234068) node {$v$};
\draw [fill=black] (4.6730953818592225,-5.655774090714479) circle (2.0pt);
\draw[color=black] (4.837021851177937,-5.265719047234068) node {$w$};
\draw [fill=black] (4.6730953818592225,-3.6557740907144787) circle (2.0pt);
\draw[color=black] (4.837021851177937,-3.272563479526905) node {$z$};
\draw [fill=black] (0.6730953818592189,-3.6557740907144787) circle (2.0pt);
\draw[color=black] (0.8507107157636351,-3.272563479526905) node {$x$};
\draw [fill=black] (6.015979730307655,-5.62951741398196) circle (2.0pt);
\draw[color=black] (5.714959422667992,-5.360631217124886) node {$u$};
\draw [fill=black] (8.015979730307663,-5.62951741398196) circle (2.0pt);
\draw[color=black] (8.182675839829226,-5.241991004761364) node {$v$};
\draw [fill=black] (10.01597973030766,-5.62951741398196) circle (2.0pt);
\draw[color=black] (10.175831407536377,-5.241991004761364) node {$w$};
\draw [fill=black] (10.01597973030766,-3.6295174139819597) circle (2.0pt);
\draw[color=black] (10.175831407536377,-3.248835437054201) node {$z$};
\draw [fill=black] (6.015979730307655,-3.6295174139819597) circle (2.0pt);
\draw[color=black] (6.189520272122075,-3.248835437054201) node {$x$};
\draw [fill=black] (7.9928515000475935,-7.021594190214189) circle (2.0pt);
\draw[color=black] (8.0,-7.3) node {$y$};
\end{tikzpicture}\caption{Families of digraphs that are not $\lambda'$-connected. Dotted line indicates adjacency.}
\end{center}
\end{figure}
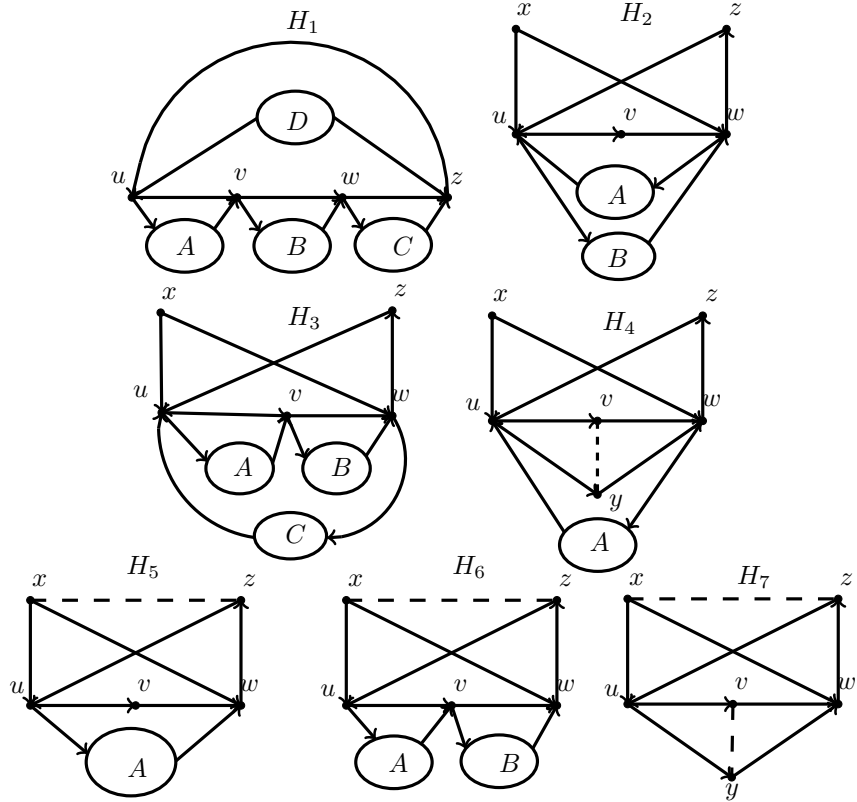

Observe that by Theorem \ref{characterization} it follows that the digraphs of $H_1,H_2,\dots ,H_7$ are not $\lambda'$-connected. 

\begin{theorem}
Let $D$ be a  strong digraph of girth $4$ and $|V(D)|\ge 6$. If $D$ is not isomorphic to a member of the families $H_1,H_2,\dots ,H_7$, then $D$ is $\lambda'$-connected and 
$$
\lambda(D)\le \lambda'(D)\le \xi(D).
$$
\end{theorem}

\begin{proof}
First we prove the right hand of the inequality. Let $C=(u,v,w,z,u)$ be a $4$-cycle of $D$ such that $\xi(D)=\xi(C)$. Suppose without loss of generality that  $\xi(C)=d^{+}(u)+d^{+}(v)+d^{+}(w)+d^{+}(z)-4$.  If $D-\{u,v,w,z\}$ contains an arc, then $D$ is $\lambda'$-connected and $\lambda'(D)\le\xi(D)$. Otherwise  $D-\{u,v,w,z\}$ consists of a set of isolated vertices.  Since $D$ is not isomorphic to a member of  $H_{1}$,  $D$ must contains a $4$-cycle $C'$ containing two arcs of $C$. Let $C'=(u,v,w,x,u)$. We continue the proof by distinguishing three cases.

\paragraph{\textbf{{Case 1}}}  Assume that  \emph{$d^{+}(x)=d^{-}(x)=1$}.

\paragraph{\textbf{{Subcase 1.1}}} If $d^{+}(z)=d^{-}(z)=1$. Since $D$ is not isomorphic to any member of $H_{2}$, $H_{3}$ and $H_{4}$, it follows that $|V(D)|\ge 7$ impliying that  there exists a set of vertices 
$a_{1},a_{2},\dots , a_{m}$, $m\ge 2$,  such that $a_{i}\notin \{u,v,w,x,z\}$ for $1\le i\le m$. If $d^{+}(v)=d^{-}(v)=1$. Since $D$ is strong, it follows that $d^{+}(a_{i})=d^{-}(a_i)=1$ for every $1\le i \le m$ impliying that $D$ is isomorphic to a member of $H_{2}$, a contradiction. 
Therefore, either $d^{+}(v)\ge 2$ or  $d^{-}(v)=2$.
If $d^+(v)\ge 2$ and $d^-(v)=1$, then there exists a vertex $a_1$ such that $v \rightarrow a_1$. Since $D$ is strong and has girth 4, it follows that $a_1\rightarrow w$. Moreover, since $D$ is  not a member of the families $H_{3}$ and $H_{4}$,  there exists a vertex $a_2$, $a_2\neq a_1$ such that $u \rightarrow a_2 \rightarrow w$. Also, as $d^{-}(v)=1$, it follows that $d^{+}(a_2)=1$.  Consider the $4$-cycle $C_{1}=(u,a_{2},w,z,u)$, therefore
\begin{eqnarray*}
\xi(C_{1}) & \le & d^{+}(u)+d^{+}(a_{2})+d^{+}(w)+d^{+}(z)-4\\
& < & d^{+}(u)+d^{+}(v)+d^{+}(w)+d^{+}(z)-4 = \xi(D),
\end{eqnarray*}
giving a contradiction.

If $d^+(v) =1$ and $d^-(v)\ge 2$, then there exists a vertex $a_1$ such that $a_1 \rightarrow v$. Further, since $D$ is strong and has girth 4, it follows that  $u\rightarrow a_1$. 
As $D$  is not isomorphic to any member of families $H_3$ and $H_4$,  there exists a vertex $a_2$, $a_2\neq a_1$ such that $u \rightarrow a_2\rightarrow w$. Let $S=\{ua_1,vw,wx\}\subset A(D)$. The digraph $D-S$ has a strong component $D_1$ containing the $4$-cycle $(u,a_2,w,z,u)$ and $D-S$  contains the arc $a_1v$. 
Therefore $D$ is $\lambda'$-connected and 
$$
\lambda'(D)\le |S| \le d^{+}(u)+d^{+}(v)+d^{+}(w)+d^{+}(z)-4=\xi(D).
$$
If $d^+(v)\ge 2$ and $d^-(v)\ge 2$, then there exists two vertices $a_1,a_2$, such that $a_1\rightarrow v$ and $v\rightarrow a_2$. Since $D$ is strong and has girth 4,  $u\rightarrow a_1$ and $a_2\rightarrow w$.  Since $D$  is not isomorphic to any member of  the family  $H_3$, there exists a vertex $a_3$  such that $u\rightarrow a_3\rightarrow w$. 
Let $S=\partial^+(\{u,a_3,w,z\})$, then $S$ is a restricted arc-cut of $D$ and
\begin{eqnarray*}
\lambda'(D)& \le &  |S|\le d^+(u)+d^+(a_3)+d^+(w)+d^+(z)-4\\
& \le & d^+(u)+2+d^+(w)+d^+(z)-4\\
& \le & d^+(u)+d^+(v)+d^+(w)+d^+(z)-4\\
& = & \xi(D),
\end{eqnarray*}
giving a contradiction.

\paragraph{\textbf{{Subcase 1.2}}}  Assume that either $d^+(z)\ge 2$ or $d^-(z)\ge 2$. This implies that there  exists a vertex $a$, different from $u,w,x$. Suppose that $z \rightarrow a$. Therefore
\begin{eqnarray*}
\xi((u,v,w,x,u)) & \le & d^+(u)+d^+(v)+d^+(w)+d^+(x)-4\\
& < &  d^+(u)+d^+(v)+d^+(w)+2-4=\xi(D).\\
\end{eqnarray*}

Continue assuming that $a\rightarrow z$. Let  $S=\partial^+(\{u,v,w,x\})$. Note that $D-S$  has a  strong component $D_1$ containing
the $4$-cycle $(u,v,w,x,u)$ and $D-V(D_1)$ contains the arc $az$. Hence $S$ is a $\lambda'$-restricted arc cut and 
\begin{eqnarray*}
\lambda'(D) \le |S| & \le & d^+(u)+d^+(v)+d^+(w)+d^+(x)-4\\
& = &  d^+(u)+d^+(v)+d^+(w)+1-4\\
& \le &  d^+(u)+d^+(v)+d^+(w)+d^+(z)-4=\xi (D),
\end{eqnarray*}
 and the result follows.
 
\paragraph{\textbf{{Case 2}}} Assume that $d^+(x)=1$ and $d^-(x)=2$. This implies that $z\rightarrow x$ and therefore $d^+(z)\ge 2$. Since $(u,v,w,x,u)$ is a $4$-cycle,  it follows that 
\begin{eqnarray*}
\xi((u,v,w,x,u)) & \le & d^+(u)+d^+(v)+d^+(w)+d^+(x)-4\\
& < &  d^+(u)+d^+(v)+d^+(w)+2-4\\
& \le &  d^+(u)+d^+(v)+d^+(w)+d^+(z)-4=\xi(D),\\
\end{eqnarray*}
 yielding a contradiction. 
 
 \paragraph{\textbf{{Case 3}}} Assume that $d^+(x)=2$ and $d^-(x)=1$.  This implies that $x\rightarrow z$.
 
 \paragraph{\textbf{{Subcase 3.1}}} If $d^+(z)=1$ and $d^-(z)=2$. 
 Suppose that $d^+(v)=d^-(v)=1$. Since $D$ is not isomorphic to any member of  the family  $H_5$, it follows that there exits a vertex $a_1$ such that $w\rightarrow a_1 \rightarrow u$. Let $S=\partial^+(\{u,v,w,a_1\})$.  The digraph $D-S$ has a strong component $D_1$ containing
 the $4$-cycle  $(u,v,w,a_1,u)$  and $D-V(D_1)$ contains the arc $xz$. Hence $D$ is $\lambda'$-connected and 
 \begin{eqnarray*}
\lambda'(D) & \le & d^+(u)+d^+(v)+d^+(w)+d^+(a_1)-4\\
& = &  d^+(u)+d^+(v)+d^+(w)+1-4\\
& = &  d^+(u)+d^+(v)+d^+(w)+d^(z)-4=\xi(D).\\
\end{eqnarray*}
 
 Continue assuming that either $d^+(v)\ge 2$ or $d^-(v)=2$. If $d^+(v)\ge 2$ and $d^-(v)=1$, then there exists a vertex $a_1$ such that
 $v\rightarrow a_1$. Further, as $D$ is strong, it follows that $a_1\rightarrow w$. Since $D$ is not isomorphic to any member of  the families  $H_6$ and  $H_7$, the order of $D$ is at least $7$ and there exists a vertex $a_2$ adjacent to $u$ and $w$. If $u\rightarrow a_2 \rightarrow w$, then $a_2$ is not adjacent to $v$ and  $d^+(a_2)=1$. Since $(u,a_2,w,z,u)$ is a $4$-cycle, it follows that 
  \begin{eqnarray*}
\xi((u,a_2,w,z,u))  & \le & d^+(u)+d^+(a_2)+d^+(w)+d^+(z)-4\\
& = &  d^+(u)+1+d^+(w)+d^+(z)-4\\
& < &  d^+(u)+d^+(v)+d^+(w)+d^(z)-4=\xi(D),
\end{eqnarray*}
giving a contradiction.
 
 Suppose that  $w\rightarrow a_2 \rightarrow u$. Let $S=\partial^+(\{u,v,w,a_2\})$. The digraph
 $D-S$  has a strong component $D_1$ containing the $4$-cycle $(u,v,w,a_2,u)$ and $D-V(D_1)$ has the arc $xz$. Therefore $D$ is $\lambda'$-connected and 
 
\begin{eqnarray*}
\lambda'(D)  & \le & d^+(u)+d^+(v)+d^+(w)+d^+(a_2)-4\\
& = &  d^+(u)+d^+(v)+d^+(w)+1-4\\
& = &  d^+(u)+d^+(v)+d^+(w)+d^+(z)-4=\xi(D).
\end{eqnarray*}

If $d^+(v)=1$ and $d^-(v)\ge 2$, then there exists a vertex $a_1$ such that $a_1\rightarrow v$, and since $D$ is 
strong it follows that $u\rightarrow a_1$.  Since $D$ is not isomorphic to any member of  the families  $H_6$ and  $H_7$, then $|V(D)|\ge 7$ and there exists a vertex $a_2$ such that $a_2$ an $w$ are adjacent. 
Suppose that $u\rightarrow a_2\rightarrow w$. If $S=\{ua_1,vw,wx\}\subset A(D)$, then $D-S$  has a strong compononent $D_1$ containing the $4$-cycle $(u,a_2, w,z,u)$and $D-V(D_1)$ has the arc $a_1v$. 
Therefore $D$  is $\lambda'$-connected and 
$$
\lambda'(D)   \le  3 \le d^+(u)+d^+(v)+d^+(w)+d^+(z)-4=\xi(D).
$$ 
Continue assuming that $w\rightarrow a_2\rightarrow u$.
If $S=\partial^+(\{u,v,w,a_2\})\subset A(D)$, then $D-S$ is a restricted arc cut of $D$ such that $D-S$ has a strong component $D_1$ containing the $4$-cycle $(u,v,w,a_2,u)$ and $D-V(D_1)$ has the arc $xz$. Therefore,
\begin{eqnarray*}
\lambda'(D)  & \le &  |S| =  d^+(u)+d^+(v)+d^+(w)+d^+(a_2)-4\\
& = &  d^+(u)+d^+(v)+d^+(w)+1-4\\
& = &  d^+(u)+d^+(v)+d^+(w)+d^+(z)-4=\xi(D).
\end{eqnarray*}
If $d^+(v)\ge 2$ and $d^-(v)\ge 2$, then  there are two vertices $a_1$ and $a_2$ such that $a_1\rightarrow v$ and  
$v\rightarrow a_2$. Since $D$ is strong and has girth $4$ it follows that $u\rightarrow a_1$ and $a_2\rightarrow w$ (may be the case that $a_1\rightarrow w$ or $u\rightarrow a_2$). Since $D$ is not isomorphic to any member of the family $H_6$ there exists a vertex $a_3$ adjacent to $u$ and $w$.
Suppose that $u\rightarrow a_3\rightarrow w$ (may be the case where $a_3=a_1$ or $a_3=a_2$ or $a_3$ is adjacent to $w$).  Let $S=\partial^+(\{u,a_3, w, z\})$. The digraph $D-S$ has a strong component $D_1$ containing de $4$-cycle $(u,a_3,w,z,u)$ and $D-V(D_1)$ has the arc $a_1v$ or $va_2$, according to the case. Therefore $D$ is $\lambda'$-connected and 
\begin{eqnarray*}
\lambda'(D)  & \le &  |S| =  d^+(u)+d^+(a_3)+d^+(w)+d^+(z)-4\\
& \le &  d^+(u)+2+d^+(w)+d^+(z)-4\\
& \le &  d^+(u)+d^+(v)+d^+(w)+d^+(z)-4=\xi(D).
\end{eqnarray*}

Suppose that $w\rightarrow a_3\rightarrow u$. If $S=\partial^+(\{u,v,w,a_3\})\subset A(D)$, then the digraph $D-S$  has a strong component $D_1$ containing de $4$-cycle $(u,v,w,a_3,u)$ and $D-V(D_1)$ has the arc $xz$. Therefore $D$ is $\lambda'$-connected and 
\begin{eqnarray*}
\lambda'(D)  & \le &  |S| =  d^+(u)+d^+(v)+d^+(w)+d^+(a_3)-4\\
& = &  d^+(u)+2+d^+(w)+1-4\\
& = &  d^+(u)+d^+(v)+d^+(w)+d^+(z)-4=\xi(D).
\end{eqnarray*}
\paragraph{\textbf{{Subcase 3.2}}} If $d^+(z)\ge 2$ or $d^-(z)\ge 3$.
Then there exists a vertex $a\notin \{u,w,x\}$ such that $a$ and $z$ are adjacent.  If $z\rightarrow a$, then consider the set of arcs $S=\partial^+(\{u,v,w,x\})$. Therefore the digraph  $D-S$ has a strong component $D_1$ containing de $4$-cycle $(u,v,w,x,u)$ and $D-V(D_1)$ has the arc $az$. 
Consequently, $D$ is $\lambda'$-connected and 
\begin{eqnarray*}
\lambda'(D)  & \le &  |S| =  d^+(u)+d^+(v)+d^+(w)+d^+(x)-4\\
& = &  d^+(u)+d^+(v)+d^+(w)+2-4\\
& \le &  d^+(u)+d^+(v)+d^+(w)+d^+(z)-4=\xi(D).
\end{eqnarray*}
\end{proof}
Now, suppose that $a\rightarrow z$. Since $D$ is strong it follows that either $v\rightarrow a$ or $w\rightarrow a$. If $v\rightarrow a$, let
$S=\partial^+\{u,v,a,z\}$. Therefore $D-S$ has a strong component $D_1$ containing de $4$-cycle $(u,v,a,z,u)$ and $D-V(D_1)$ has the arc $wx$. Therefore $D$ is $\lambda'$-connected and 
\begin{eqnarray*}
\lambda'(D)  & \le &  |S| =  d^+(u)+d^+(v)+d^+(a)+d^+(z)-4\\
& \le &  d^+(u)++d^+(v)+2+d^+(z)-4\\
& \le &  d^+(u)+d^+(v)+d^+(w)+d^+(z)-4=\xi(D).
\end{eqnarray*}
Continue assuming that $w\rightarrow a$. If either $v\rightarrow a$ or there exists a vertex $a'\neq a$ such that $z\rightarrow a'$, then this case is reduced to one of the two previous subcases. Otherwise observe that the condition on the girth implies that neither $a\rightarrow v$ nor the other $u\rightarrow a$. Suppose that $a\rightarrow u$. If $S=\{zu,au\}$, then the digraph $D-S$ has a strong component $D_1$ containing de $4$-cycle $(u,v,w,x,u)$ and $D-V(D_1)$ has the arc $yz$. Therefore $D$ is $\lambda'$-connected and 
\begin{eqnarray*}
\lambda'(D) \le  2   \le  d^+(u)+d^+(v)+d^+(w)+d^+(z)-4=\xi(D).
\end{eqnarray*}
Finally, to prove the left inequality, since every restricted cut is a cut, it follows that $\lambda(D)\le \lambda'(D)$.

  \end{document}